\documentclass[a4paper]{amsart}

\usepackage{amsaddr}% addres in fisrt page (only for amsart)

\usepackage{svninfo}

\usepackage{amsmath, amsthm, amsfonts, amssymb, latexsym, mathrsfs,stmaryrd}
\usepackage{mathtools}
\usepackage{accents}
\usepackage{cancel} %%define nomdels
\usepackage{hyperref}
\usepackage[utf8]{inputenc}
\usepackage{booktabs,colortbl}

\usepackage{enumitem}
\usepackage{booktabs,colortbl}
\usepackage{multicol,xcolor}
\usepackage{graphicx}
\usepackage{tikz}
\usetikzlibrary{positioning}
\usepackage[font=small,labelfont=bf]{caption}  %for captionof
\usepackage{hyperref}
\usepackage{url}

\renewcommand{\today}{%\number\day\space
  \ifcase\month\or
  January\or February\or March\or April\or May\or June\or
  July\or August\or September\or October\or November\or December\fi
  \space \number\year}

\newenvironment{dotto}{\buildrel\textstyle.\over{\hbox{\vrule height3pt depth0pt width0pt}{\smash\to}}}

\newenvironment{dotv}{\buildrel\textstyle.\over{\hbox{\vrule height3pt depth0pt width0pt}{\smash\vee}}}

\newenvironment{dota}{\buildrel\textstyle.\over{\hbox{\vrule height-2pt depth0pt width0pt}{\smash\wedge}}}

\newenvironment{g}{G\"{o}del~}

\newenvironment{lo}{{\L}ukasiewicz~}

\newenvironment{too}{\rightarrowtail}

\newtheorem{theorem}{Theorem}[section]

\newtheorem{lemma}[theorem]{Lemma}

\theoremstyle{definition}
\newtheorem{definition}[theorem]{Definition}

\newtheorem{remark}[theorem]{Remark}
\newtheorem{example}[theorem]{Example}

\DeclareSymbolFont{AMSb}{U}{msb}{m}{n}
\DeclareMathSymbol{\N}{\mathbin}{AMSb}{"4E}
\DeclareMathSymbol{\Z}{\mathbin}{AMSb}{"5A}
\DeclareMathSymbol{\R}{\mathbin}{AMSb}{"52}
\DeclareMathSymbol{\Q}{\mathbin}{AMSb}{"51}
\DeclareMathSymbol{\I}{\mathbin}{AMSb}{"49}
\DeclareMathSymbol{\C}{\mathbin}{AMSb}{"43}

\makeatletter

\def\dotminussym#1#2{%
  \setbox0=\hbox{$\m@th#1-$}%
  \kern.5\wd0%
  \hbox to 0pt{\hss\hbox{$\m@th#1-$}\hss}%
  \raise.8\ht0\hbox to 0pt{\hss$\m@th#1.$\hss}%
  \kern.5\wd0}

\makeatother

\title{A metric-like topology on BL-algebras}

\author{Seyed Mohammad Amin Khatami}
\address{Department of Computer Science, Birjand University of Technology, Birjand, Iran.}
\email{khatami@birjandut.ac.ir} \urladdr{}
\begin{document}
\maketitle
\begin{abstract}
This paper is devoted to introduce a topology on BL-algebras, makes them
semitopological algebras. For any BL-algebra
$\mathcal{L}=(L, \wedge, \vee, *, \too , 0, 1)$, the introduced topology is defined by a distance-like function between elements of $L$ which is defined by $a \leftrightarrow b=(a\too b)*(b\too a)$. We will show that when the continuous scale $[0,1]$ is endowed to be a BL-algebra, then this topology admits some of the most important properties of the metric topology. Finally, we will show that this topology can be examined by a similar topology on dual of BL-algebras as well.
\smallskip
\noindent\emph{Keywords:} BL-algebra, dual of BL-algebra, topological BL-algebra, open ball topology, similarity topology
\end{abstract}
\section{Introduction}\label{sec:1}
Triangular norms and triangular conorms, shortly t-norms and s-norms, have been used in several areas of mathematics. Their origin have been goes back to \cite{menger,sch}. One of the areas that t-norms and s-norms have been appeared, is many-valued logics. Indeed, a t-norm (s-norm)
is a generalization for the interpretation of the conjunction connective (disjunction connective) \cite{hol,als}.

Basic logic introduced by H{\'a}jek in the early of 1998 \cite{hajek98}, is known as the logic of continuous t-norms. The algebraic counterpart
of a propositional basic logic is a BL-algebra.
%The Lindenbaum algebra of equivalent formulas of propositional basic logic is a BL-algebra.
MV-algebras, introduced by Chang \cite{chang1959new} to prove the completeness theorem for \lo logic, are special types of BL-algebras. A more general algebraic structure originated in logics without contractions is residuated lattice. The oldest version of such structure appeared in classical logic is Boolean algebra.

Algebraic structures are studied in algebraic and topological point of view. Algebra studies the property of operations and algorithmic computations of a space, while topology provide a framework for understanding its geometric properties.
Besides introducing the concept of BL-algebras \cite{hajek98}, their algebraic and topological properties are of the most interesting research areas.

Bozooei et.al in \cite{zah-bor-2016, bor-rez-kou-2011} introduced the notion of topological BL-algebras. In \cite{bor-rez-kou-2012} they studied the metrizability of BL-algebras as well.
The aim of this article is to introduce a metric-like topology on BL-algebras which makes them semitopological algebras in the sense of \cite{bor-rez-kou-2011}.

One of the biggest obstacles of extending the results of classical logic
to basic logic is non-continuity of the interpretations of logical connectives. Therefore,
the mentioned topology on BL-algebras could be seen as an applicable tool to extend the results of
classical logic to H{\'a}jek Basic logic.

Here, for any BL-algebra $\mathcal{L}=(L, \wedge, \vee, *, \too , 0, 1)$, we define two topologies
$T_*$ and $\mathbf{T}_*$ on $L$ and $L^2$ that all the operators of $\mathcal{L}$ becomes continuous function with respect to these topologies. The construction of $T_*$ is based on the $*$-balls $B_r(a)=\{b: a\leftrightarrow b\ge r\}$ in which $a\leftrightarrow b=(a\too b)*(b\too a)$ and so it seems like a metric topology. We show that when the continuous scale $[0,1]$ is endowed to be a BL-algebra, $T_*$ admits some of the most important properties of the metric topology.
This fact results in a simpler way for analysing $T_*$. Indeed, we show that when $[0,1]$ is considered to be a BL-algebra, then $T_*$ could be examined by a topology $T_\star$ on $[0,1]$ as a dual of a BL-algebra. The studying of $T_\star$ on $[0,1]$ as a dual of a BL-algebra in the cases that $T_\star$ forms a metric topology, has been the subject of the author conference paper \cite{khatami2018}.

The rest of the paper organized as follows: Section 2 presents a summary of t-norms, s-norms and BL-algebras. Section 3 introduces a topology $T_*$ on any BL-algebra $\mathcal{L}=(L, \wedge, \vee, *, \too , 0, 1)$ which makes it a semitopological algebra. Section 4 shows when $[0,1]$ is considered to be a BL-algebra, $T_*$ admits some the important properties of the metric topology.
Finally, Section 5 defines a dual concept for BL-algebras and examining $T_*$ on $[0,1]$ by a topology on its dual.
\section{Preliminaries}\label{sec:per}
In many-valued logics, t-norms and s-norms sometimes play the role of the interpretation of
the conjunction and disjunction connective. Recall that a triangular norm, in shortly a t-norm, is a binary function $T$ from $[0,1]^2$ into $[0,1]$ which is associative, commutative, non-decreasing on both arguments and $T(1,x)=x$ for all $x\in[0,1]$ \cite[Definition 1.1]{Klement}. The concept of t-conorm or s-norm reversed the boundary condition of the concept of t-norm. Thus an s-norm is an associative, commutative, and non-decreasing function $S$ from the unite square into the unite interval satisfying for all $x\in[0,1]$ the boundary condition $S(0,x)=x$ \cite[Definition 1.13]{Klement}.

Bellow, the most important t-norms and s-norms which are employed in the most significant many-valued logics as conjunction and disjunction are listed in Table \ref{tnor}.
\begin{table}[ht!]
\centering
%\vspace{0.4cm}
%\vspace{-0.9cm}
\begin{tabular}{l l l}
\toprule
\rowcolor{black!20}
t-norm & s-norm
\\
\midrule
$T_L(x,y)=\max\{0,x+y-1\}$&$S_L(x,y)=\min\{1,x+y\}$
\\
\midrule
$T_G(x,y)=\min\{x,y\}$&$S_G(x,y)=\max\{x,y\}$
\\
\midrule
$T_\pi(x,y)=x.y$&$S_\pi(x,y)=x+y-x.y$
\\
\midrule
\end{tabular}
%�\smallskip�
%\begin{fontspace}{0.75}{2}
\caption{\lo, \g, and Product t-norm and s-norm}\label{tnor}
%\end{fontspace}
\end{table}

In 1998, H{\'a}jek introduced a many-valued logic called  {\em Basic logic} based on arbitrary continuous t-norms \cite{hajek98}. Indeed, Basic logic could be seen as an extension of the \lo, \g, and Product logic.

Assume that $T$ is a continuous t-norm and $R_T$ is its residua which is defined by
\begin{equation}\label{res}
z\le R_T(x,y)~~~~~ \text{iff} ~~~~~T(z,x)\le y
\end{equation}
for all $x,y,z\in [0,1]$. If $Prop$ is generated from a set of atomic propositions $P$ by formal operations $\{\&, \to, \bot\}$ and $e_0: P\to[0,1]$ is a function, there is a unique extension $e$ of $e_0$, called an evaluation, satisfying the following rules \cite[Section 2.2]{hajek98}:
\begin{itemize}
  \item $e(\bot)=0$,
  \item $e(\varphi \&\psi)=T\left(e(\varphi),e(\psi)\right)$,
  \item $e(\varphi\to \psi)=R_T(e(\varphi),e(\psi))$.
\end{itemize}
The algebraic counter part of a theory in Basic logic, forms an algebra, called BL-algebra. Indeed,
if for a theory $\Sigma\subseteq Prop$, we define
\begin{itemize}
\item $[\varphi]=\{\psi: T\vdash\varphi\leftrightarrow\psi\}$,
\item $Lind(\Sigma)=\{[\varphi]: \varphi\in Prop\}$,
\item $[\varphi]\le[\psi]$ iff $\Sigma\vdash(\varphi\to\psi)$,
\item $[\top]=[\bot\to\bot]$,
\item $[\varphi]*[\psi]=[\varphi\&\psi]$,
\item $[\varphi]\rightarrowtail[\psi]=[\varphi\to\psi]$,
\end{itemize}
then, $(Lind(\Sigma),\le, *, \rightarrowtail, [\bot],[\top])$ forms a BL-algebra \cite[Lemma 2.3.12]{hajek98}. Actually,
we have the following definition for a BL-algebra.
\begin{definition}\cite[Deinition 2.3.3]{hajek98}
A \emph{BL-algebra} is an algebra $\mathcal{L}=(L, \wedge, \vee, *, \too , 0, 1)$ of type $(2,2,2,2,0,0)$ satisfying the following properties:
\begin{enumerate}[label={\normalfont (BL\arabic*)}]
  \item $(L, \wedge, \vee, 0 ,1)$ is a bounded lattice with the greatest element $1$ and the smallest element $0$,\label{bl1}
  \item $(L, *, 1)$ is an Abelian monoid,\label{bl2}
  \item $\too$ is the residua of $*$, i.e., $c\le a\too b$ iff $c*a\le b$ for all $a,b,c\in L$,\label{bl3}
  \item $a\wedge b=a*(a\too b)$ for all $a,b\in L$,\label{bl4}
  \item $(a\too b)\vee(b\too a)=1$ for all $a,b\in L$.\label{bl5}
\end{enumerate}
\end{definition}
For any continuous t-norm $T$ and its residua $R_T$,
\begin{center}
$[0,1]_T=([0,1], \min, \max, T, R_T, 0, 1)$
\end{center}
forms a BL-algebra \cite[Chapter 2]{hajek98}.
Conversely, when the continuous scale $[0,1]$ endowed to be a BL-algebra, the binary operator $*$ becomes a continuous t-norm on $[0,1]$ \cite{godo99}. The standard BL-algebra on the real segment $[0,1]$ which is defined by  the continuous t-norm $*$, is denoted by $[0,1]_*$.

The following fact, used several times in the outcome results of the paper. Its proof can be found in \cite[Chapter 2]{hajek98}.
\begin{lemma}\label{blprop}
Let $\mathcal{L}=(L, \wedge, \vee, *, \too , 0, 1)$ be
a BL-algebra. The following properties holds in $\mathcal{L}$.
 %\begin{multicols}{2}
%\begin{enumerate}[label={\thelemma.\arabic*}]
\begin{enumerate}[label={\normalfont (B\arabic*)}]
\item\label{b1} $a*b=b*a$ and $(a*b)*c=a*(b*c)$,\label{f1}
\item\label{b3} $a*(a\too b)\le b$ and $a\le b\too(a*b)$,\label{f3}
\item\label{b4} $a\le b$ iff $a\too b=1$,\label{f4}
\item\label{b5} if $a\le b$ then $a*c\le b*c$, $c\too a\le c\too b$, and $a\too c\ge b\too c$,\label{f5}
\item\label{b2} $a*0=0$,\label{f2}
\item\label{b6} $(a\vee b)*c=(a*c)\vee(b*c)$,\label{f6}
\item\label{b7} $a*b\le a$ and $a\le b\too a$,\label{f7}
\item\label{b8} $a\vee b=\big((a\too b)\too b\big)\wedge\big((b\too a)\too a\big)$,\label{f8}
\item\label{b9} $(a\too b)\le (b\too c)\too(a\too c)$,\label{f9}
\item\label{b10} $(a\too b)*(b\too c)\le(a\too c)$,\label{f10}
\item\label{b11} $a\too(b\too c)=(a*b)\too c$,\label{f11}
\item\label{b12} $a\too(b\too c)=b\too(a\too c)$,\label{f12}
\item\label{b13} $a\too a=1$,\label{f13}
\item\label{b14} $a\too b\le(a*c)\too(b*c)$,\label{f14}
\item\label{b15} $(a\too b)*(c\too d)\le(a*c)\too(b*d)$.\label{f15}
\end{enumerate}
%\end{multicols}
\end{lemma}
\section{A topology on BL-algebras makes them semitopological algebras}
In this section, a topology on arbitrary BL-algebras
introduced which makes them semitopological algebras.

From now on, we denote $(a_1,a_2)$ shortly by $\mathbf{a}$. The following crucial definition is needed for Definition \ref{top2}.
\begin{definition}\label{equi}
Let $\mathcal{L}=(L, \wedge, \vee, *, \too , 0, 1)$ be a BL-algebra.
An element $a\in L$ is called \emph{strongly less than $1$}, denoted by $a\ll 1$, whenever
for any $b\in L$, $a\vee b=1$ implies that $b=1$. Furthermore, $\leftrightarrow$ and $\Leftrightarrow$ are operators on $L$ and $L^2$ which are defined respectively as follows:
\begin{equation}\label{equ}
a \leftrightarrow b=(a\too b)*(b\too a)~~~~~,~~~~~
\mathbf{a}\Leftrightarrow\mathbf{b}=(a_1\leftrightarrow b_1)*(a_2\leftrightarrow b_2).
\end{equation}
\end{definition}
In the following lemma, some of the properties of the notions $\ll$, $\leftrightarrow$, and $\Leftrightarrow$ are established.
\begin{lemma}
Let $\mathcal{L}=(L, \wedge, \vee, *, \too , 0, 1)$ be a BL-algebra.
\begin{enumerate}[label={\normalfont (L\arabic*)}]
  \item\label{l1} $0\ll 1$.
  \item For any $a\in L$, if $a\ll 1$, then $a<1$.
  \item For any $a,b\in L$, if $b<a\ll 1$, then $b\ll 1$,
  \item\label{l4} For any $a,b\in L$, if $a\ll 1$ and $b\ll 1$, then $a\vee b\ll 1$.
  \item\label{l5} For any $a,b\in L$, $a\too b\ge a\leftrightarrow b$,
  \item\label{l6} Both of $\leftrightarrow$ and $\Leftrightarrow$ are symmetric.
  \item\label{l7} For any $a\in L$ and $\mathbf{a}\in L^2$,
  $a\leftrightarrow a=1$ and $\mathbf{a}\Leftrightarrow\mathbf{a}=1$.
  \item\label{l8} For any $a,b,c \in L$,
  $(a\leftrightarrow c)*(c\leftrightarrow b)\le a\leftrightarrow b$.
  \item\label{l9} For any $\mathbf{a}, \mathbf{b}, \mathbf{c}\in L^2$,
  $(\mathbf{a}\Leftrightarrow\mathbf{c})*(\mathbf{c}\Leftrightarrow\mathbf{b})
  \le \mathbf{a}\Leftrightarrow\mathbf{b}$.
  \item\label{l10} For any $a,b\in L$, $a\leftrightarrow b=1$ iff $a=b$.
% \item if $a\gg 0$ and $a>b$, then $b\dotto a\gg 0$.
\end{enumerate}
\end{lemma}
\begin{proof}~\\
\begin{enumerate}[label={L\arabic*})]
\item By \ref{bl1}, we know that $0\vee x=x$ for any $x\in L$. So, for any $x\in L$, $0\vee x=1$ implies that $x=1$. Therefore $0\ll 1$.
\item On the contrary, if $a=1$ then for $x\ne 1$, $a\vee x=1$, which is in contradiction with $a\ll 1$.
\item For an arbitrary $x\in L$, assume that $b\vee x=1$. Since $b<a$, $b\vee x\le a\vee x$. So,
$a\vee x=1$ which together with $a\ll 1$ implies that $x=1$. Thus $b\ll 1$.
\item For an arbitrary $x\in L$, assume that $(a\vee b)\vee x=1$. Thus,
$a\vee(b\vee x)=1$. Therefore, $b\vee x=1$. Hence, $x=1$ that is $a\vee b\ll 1$.
\item[]\hspace{-0.8cm}L5, L6, and L7) Follows respectively from \ref{b7}, \ref{b1}, and \ref{b13}.
\item[L8)] By \ref{b10} for any $a,b,c\in L$, $(a\too c)*(c\too b)\le(a\too b)$. Similarly,
$(b\too c)*(c\too a)\le(b\too a)$. Now using \ref{b5} twice together with \ref{b1} implies \ref{l8}.
\begin{align*}
  a\leftrightarrow b &=(a\too b)*(b\too a)\\
  &\ge\big((a\too c)*(c\too b)\big)*(b\too a)\\
  &\ge\big((a\too c)*(c\too b)\big)*\big((b\too c)*(c\too a)\big)\\
  &=\big((a\too c)*(c\too a)\big)*\big((b\too c)*(c\too b)\big)\\
  &=(a\leftrightarrow c)*(b\leftrightarrow c)
\end{align*}
\item[L9)] Follows immediately from \ref{l8} together with \ref{b1}.
\begin{align*}
  \mathbf{a}\Leftrightarrow\mathbf{b}&=(a_1\leftrightarrow b_1)* (a_2\leftrightarrow b_2)\\
  &\ge\big((a_1 \leftrightarrow c_1)* (c_1 \leftrightarrow b_1)\big)*
  \big((a_2\leftrightarrow c_2)* (c_2\leftrightarrow b_2)\big)\\
  &=\big((a_1 \leftrightarrow c_1)* (a_2 \leftrightarrow c_2)\big)*
  \big((c_1\leftrightarrow b_1)* (c_2\leftrightarrow b_2)\big)\\
  &=(\mathbf{a}\Leftrightarrow\mathbf{c})*(\mathbf{c}\Leftrightarrow\mathbf{b}).
\end{align*}
\item[L10)]
One direction is obvious from \ref{l7}. For the other direction, if $a\leftrightarrow b=1$, then
$(a\too b)*(b\too a)=1$. So, by \ref{l5}
$a\too b=1$ and $b\too a=1$. Now, by \ref{b4} $b\le a$ and $a\le b$ which implies that $a=b$.
\end{enumerate}
\end{proof}
Now, the expected topology on BL-algebras which makes them topological algebras is as follows.
\begin{definition}\label{top2}
Let $\mathcal{L}=(L, \wedge, \vee, *, \too , 0, 1)$ be a BL-algebra.
For
any elements $a,r\in L$ that $r\ll 1$, the \emph{$*$-ball around $a$ of radius $r$} is the set
\begin{center}
$B_r(a)=\{b\in L: a\leftrightarrow b>r\}.$
\end{center}
Similarly the \emph{$\mathbf{*}$-ball around $\mathbf{a}\in L^2$ of radius $r\ll 1$} is the set
\begin{center}
$\mathbf{B}_r(\mathbf{a})=\{\mathbf{b}\in L^2: \mathbf{a}\Leftrightarrow\mathbf{b}>r\}.$
\end{center}
A subset $G$ of $L$ is called an $*$-open set if for every $a\in G$ there exists a radius $r\ll 1$ such that $B_r(a)\subseteq G$. $*$-open subsets of $L^2$ defined similarly.
\end{definition}
\begin{remark}\label{r2}
By \ref{l7}, $a\in B_r(a)$ and similarly $\mathbf{a}\in\mathbf{B}_r(\mathbf{a})$. Moreover,
if $r\ge s$ then $B_r(a)\subseteq B_s(a)$ and $\mathbf{B}_r(\mathbf{a})\subseteq \mathbf{B}_s(\mathbf{a})$.
\end{remark}
\begin{theorem}\label{tT}
With the notations in Definition \ref{top2}, the family of all $*$-open subsets of $L$
form a topology on $L$ denoted by $T_*$, called the "open ball topology". Similarly $\mathbf{T}_*=\{A : A~\text{is an}~*\text{-open subset of}~L^2\}$ is a topology on $L^2$.
\end{theorem}
\begin{proof}
Obviously $\emptyset, L\in T_*$. Assume that $A, B\in T_*$. If $a\in A\cap B$, then since $A$ and $B$ are $*$-open sets, there exist
$r_A\ll 1$ and $r_B\ll 1$ such that $B_{r_A}(a)\subseteq A$ and $B_{r_B}(a)\subseteq B$. By \ref{l4},  $r=r_A\vee r_B\ll 1$. Since $r\ge r_A$, Remark \ref{r2} implies that $B_r(a)\subseteq B_{r_A}(a)$.
Similarly, $B_r(a)\subseteq B_{r_B}(a)$. Thus $B_r(a)\subseteq B_{r_A}(a)\cap B_{r_B}(a)\subseteq A\cap B$. Hence, $A\cap B$ is an $*$-open set.

Now, let $\{G_i\}_{i\in I}$ be a family of $*$-open sets and $G=\cup_{i\in I}G_i$.
If $G$ is empty there is noting to prove. Assume that $G\ne\emptyset$ and $a\in G$. So, there is
$i\in I$ such that $a\in G_i$. Since $G_i$ is an $*$-open set, there exists $r\ll 1$ such that
$B_r(a)\subseteq G_i\subseteq G$. Thus $G$ is an $*$-open set.

The second part will be proved by a similar argument.
\end{proof}
The following examples describe the introduced topology on BL-algebras more precisely.
\begin{example}\label{luk}
Let $L=[0,1]$, $a * b=\max\{0,a+b-1\}$ which is the \lo t-norm. Then, by residuation relation \ref{res} one can verify that
$a \rightarrowtail b=\min\{1, 1+b-a\}$ \cite[Theorem 2.1.8]{hajek98}. To calculate $a\leftrightarrow b$, if $a\le b$ then $b-a\ge 0$ and therefore $a \rightarrowtail b=1$ and $b \rightarrowtail a=1+a-b$ and therefore $a\leftrightarrow b=1*(1+a-b)=1+a-b$. Similarly, if $b\le a$, then
$a\leftrightarrow b=1+b-a$. Consequently, we have $a\leftrightarrow b=1-|a-b|$. In addition, $[0,1]$ is a linearly ordered BL-algebra and therefore $a\ll b$ and $a<b$ have the same meaning. So, for any $a\in[0,1]$ and $r<1$,
\begin{eqnarray*}
B_r(a) & = & \{b: a\leftrightarrow b>r\}\\
& = & \{b\in[0,1]: 1-|a-b|>r\}\\
& = & \{b\in[0,1]: |a-b|<1-r\}\\
& = & \left(a-(1-r),a+(1-r)\right)\cap[0,1]\\
& = & \left(a-(1-r),a+(1-r)\right).
\end{eqnarray*}
For example $B_{0.5}(a)=(a-0.5,a+0.5)$ and
$B_{0.2}(a)=(a-0.8,a+0.8)$ and
$B_{0.7}(a)=(a-0.3,a+0.3)$.
Verily, the $*$-ball $B_r(a)$ in $T_*$ is the open ball around $a$ of radius $1-r$ in the Euclidean topology and therefore $T_*$ is equivalent to the Euclidean topology on $[0,1]$.
\end{example}
\begin{example}\label{e2}
Let $L=\{0,a,b,c,1\}$. Define $*$ and
$\rightarrowtail$ on $L$ as follows.
\vspace{0.1cm}
\begin{center}
\begin{minipage}{1.5in}
%\definecolor{Gray}{gray}{0.85}
%\definecolor{LightCyan}{rgb}{0.88,1,1}
%\newcolumntype{a}{>{\columncolor{Gray}}c}
%\newcolumntype{b}{>{\columncolor{white}}c}
%\newcolumntype{cc}{>{$}c<{$}}
\begin{tabular}{|>{\columncolor{black!20}}>{$}c<{$}|>{$}c<{$}|>{$}c<{$}|>{$}c<{$}|>{$}c<{$}|>{$}c<{$}|}
  \hline
  \rowcolor{black!20}
  ~*~ & 0 & a & b & c & 1 \\
  \hline
  0 & 0 & 0 & 0 & 0 & 0 \\
  \hline
  a & 0 & a & c & c & a \\
  \hline
  b & 0 & c & b & c & b \\
  \hline
  c & 0 & c & c & c & c \\
  \hline
  1 & 0 & a & b & c & 1 \\
  \hline
\end{tabular}
\end{minipage}
\hspace{1cm}
\begin{minipage}{1.5in}
%\begin{tabular}{|>{\columncolor{black!20}}c|c|c|c|c|c|}
\begin{tabular}{|>{\columncolor{black!20}}>{$}c<{$}|>{$}c<{$}|>{$}c<{$}|>{$}c<{$}|>{$}c<{$}|>{$}c<{$}|}
  \hline
  \rowcolor{black!20}
  \rightarrowtail & 0 & a & b & c & 1 \\
  \hline
  0 & 1 & 1 & 1 & 1 & 1 \\
  \hline
  a & 0 & 1 & b & b & 1 \\
  \hline
  b & 0 & a & 1 & a & 1 \\
  \hline
  c & 0 & 1 & 1 & 1 & 1 \\
  \hline
  1 & 0 & a & b & c & 1 \\
  \hline
\end{tabular}
\end{minipage}
\end{center}
\vspace{0.2cm}
Obviously, $L$ is a BL-algebra. The Hasse diagram of $L$ will be illustrated in Figure \ref{L1}.
\begin{figure}[h!]
  \centering
  \begin{tikzpicture}[scale=.7]
  \node [inner sep=0pt, minimum size=2.5pt, circle, fill, draw,label=above:1] (one) at (0,2) {};
  \node [inner sep=0pt, minimum size=2.5pt, circle, fill, draw,label=left:a](a) at (-1,1) {};
  \node [inner sep=0pt, minimum size=2.5pt, circle, fill, draw,label=right:b](b) at (1,1) {};
  \node [inner sep=0pt, minimum size=2.5pt, circle, fill, draw,label=above:c](c) at (0,0) {};
  \node [inner sep=0pt, minimum size=2.5pt, circle, fill, draw,label=below:0](zero) at (0,-1) {};
  \draw (zero) -- (c) -- (a) -- (one) -- (b) -- (c);
  \end{tikzpicture}
  \caption{Hasse diagram of $L$}\label{L1}
\end{figure}

So, an easy argument leads to the following table for $\leftrightarrow$.
\vspace{0.1cm}
\begin{center}
\begin{tabular}{|>{\columncolor{black!20}}>{$}c<{$}|>{$}c<{$}|>{$}c<{$}|>{$}c<{$}|>{$}c<{$}|>{$}c<{$}|}
  \hline
  \rowcolor{black!20}
  \leftrightarrow & 0 & a & b & c & 1 \\
  \hline
  0 & 1 & 0 & 0 & 0 & 0 \\
  \hline
  a & 0 & 1 & c & b & a \\
  \hline
  b & 0 & c & 1 & a & b \\
  \hline
  c & 0 & b & a & 1 & c \\
  \hline
  1 & 0 & a & b & c & 1 \\
  \hline
\end{tabular}
\end{center}
\vspace{0.2cm}
By the Hasse diagram of $L$,
$R=\{r: r\ll 1\}=\{0,c\}$.
All $*$-balls of $T_*$, i.e. $B_r(x)=\{y: x\leftrightarrow y>r\}_{r\in R, x\in L}$, are as follows:
\begin{center}
\begin{tabular}{l l}
  $B_r(0)=\{0\}~~~\forall r\in R$ & $B_0(x)=\{a,b,c,1\}~~~\forall x>0$\\
  $B_c(a)=\{a,c,1\}$ & $B_c(b)=\{b,c,1\}$\\
  $B_c(c)=\{a,b,c\}$ & $B_c(1)=\{a,b,1\}$
\end{tabular}
\end{center}
So,
$T_*=\left\{\emptyset, \{0\}, \{a,b,c,1\}, \{0,a,b,c,1\}\right\}$.
Note that by definition of $*$-open sets, an $*$-ball is not necessarily an $*$-open set.
\end{example}
Above examples showed that the introduced topology on BL-algebras is not trivial.
Indeed, this topology is obtained from the distance between elements of a BL-algebra
with respect to the $\leftrightarrow$.

In Section 5, we will show the open ball topology on BL-algebras could be described more explicitly by a kind of duality between algebras. The topological equivalence of $T_*$ and the Euclidean topology in Example \ref{luk} could be explained by this duality as well.

Besides \ref{bl4} and \ref{b8}, the following theorem indicates that for any BL-algebra $\mathcal{L}=(L, \wedge, \vee, *, \too , 0, 1)$, all the operators of $\mathcal{L}$ are continuous functions with respect to the introduced topologies $T_*$ and $\mathbf{T}_*$ on $L$ and $L^2$.
\begin{theorem}\label{main1}
If $\mathcal{L}=(L, \wedge, \vee, *, \too , 0, 1)$ is a BL-algebra and
$T_*$ and $\mathbf{T}_*$ are the same as in Theorem \ref{tT}, then
the mappings $*:(L^2,\mathbf{T}_*)\to(L,T_*)$ and $\too:(L^2,\mathbf{T}_*)\to(L,T_*)$ would be continuous functions.
\end{theorem}
\begin{proof}
Consider an $*$-open set $A\in T_*$. We must verify that the inverse images of $A$, $*^{-1}(A)$ and $\too^{-1}(A)$ are $*$-open subsets of $L^2$.

Firstly, consider a point $\mathbf{a}\in *^{-1}(A)$, that is $a_1 * a_2\in A$. Since $A$ is
an $*$-open set, there exists $r\ll 1$ such that
$B_r(a_1 * a_2)\subseteq A$. To finalize the first part of proof, we will show that
$\mathbf{B}_r(\mathbf{a})\subseteq *^{-1}(A)$. Consider an element $\mathbf{b}\in\mathbf{B}_r(\mathbf{a})$. So, $\mathbf{a}\Leftrightarrow\mathbf{b}>r$, that is
\begin{equation}\label{a11}
(a_1\leftrightarrow b_1)*(a_2\leftrightarrow b_2)>r.
\end{equation}
In addition since by \ref{b15} we have
\begin{center}
$(a_1 * a_2)\too (b_1 * b_2)\ge (a_1\too b_1)*(a_2\too b_2)$
\end{center}
and
\begin{center}
$(b_1 * b_2)\too (a_1 * a_2)\ge (b_1\too a_1)*(b_2\too a_2)$,
\end{center}
so, applying \ref{b5} twice and then using \ref{b1} and \ref{a11} we get
\begin{align*}
&(a_1 * a_2)\leftrightarrow(b_1 * b_2)\\
&~~~~~~~~~~~~~~~=\big((a_1 * a_2)\too (b_1 * b_2)\big)
*\big((b_1 * b_2)\too (a_1 * a_2)\big)\\
&~~~~~~~~~~~~~~~\ge \big((a_1\too b_1)*(a_2\too b_2)\big)*\big((b_1 * b_2)\too (a_1 * a_2)\big)\\
&~~~~~~~~~~~~~~~\ge \big((a_1\too b_1)*(a_2\too b_2)\big)*\big((b_1\too a_1)*(b_2\too a_2)\big)\\
&~~~~~~~~~~~~~~~= \big((a_1\too b_1)*(b_1\too a_1)\big)*\big((a_2\too b_2)*(b_2\too a_2)\big)\\
&~~~~~~~~~~~~~~~= (a_1\leftrightarrow b_1)*(a_2\leftrightarrow b_2)\\
&~~~~~~~~~~~~~~~> r.
\end{align*}
Thus $b_1 * b_2\in B_r(a_1 * a_2)\subseteq A$. Hence $\mathbf{b}\in *^{-1}(A)$, completes the first part of the proof.

Secondly, to prove that $\too^{-1}(A)$ is an $*$-open subset of $L^2$, consider a point $\mathbf{a}\in \too^{-1}(A)$. So $a_1 \too a_2\in A$. Since $A$ is
an $*$-open set, there exists $r\ll 1$ that
$B_r(a_1\too a_2)\subseteq A$. To prove that
$\too^{-1}(A)$ is $*$-open, we investigate that
$\mathbf{B}_r(\mathbf{a})\subseteq \too^{-1}(A)$. To this end, if
$\mathbf{b}\in\mathbf{B}_r(\mathbf{a})$, then
\begin{equation}\label{a12}
\mathbf{a}\Leftrightarrow\mathbf{b}>r.
\end{equation}
By \ref{b9}
$a_1\too b_1\le (b_1\too b_2)\too(a_1\too b_2)$ and by \ref{bl3}
\begin{equation}\label{a22}
(a_1\too b_1)*(b_1\too b_2)\le(a_1\too b_2).
\end{equation}
Again by \ref{b9}
$a_1\too b_2\le (b_2\too a_2)\too(a_1\too a_2)$ and therefore by \ref{a22}
\begin{center}
$(a_1\too b_1)*(b_1\too b_2)\le(b_2\too a_2)\too(a_1\too a_2)$.
\end{center}
Now applying \ref{bl3} we have
\begin{center}
$\big((a_1\too b_1)*(b_1\too b_2)\big)*(b_2\too a_2)\le(a_1\too a_2)$
\end{center}
which besides \ref{b1} leads to
\begin{center}
$\big((a_1\too b_1)*(b_2\too a_2)\big)*(b_1\too b_2)\le(a_1\too a_2)$.
\end{center}
Again \ref{bl3} implies that
\begin{equation}\label{a33}
(b_1\too b_2)\too(a_1\too a_2)\ge (a_1\too b_1)*(b_2\too a_2).
\end{equation}
Analogously
\begin{equation}\label{a44}
(a_1\too a_2)\too(b_1\too b_2)\ge (b_1\too a_1)*(a_2\too b_2).
\end{equation}
Now, applying \ref{a33}, \ref{a44}, and \ref{b5} we see that
\begin{align*}
&(a_1\too a_2)\leftrightarrow(b_1\too b_2)\\
&~~~~~~~~~~~~~~~=
\big((a_1\too a_2)\too(b_1\too b_2)\big)*
\big((b_1\too b_2)\too(a_1\too a_2)\big)\\
&~~~~~~~~~~~~~~~\ge
\big((b_1\too a_1)*(a_2\too b_2)\big)*
\big((b_1\too b_2)\too(a_1\too a_2)\big)\\
&~~~~~~~~~~~~~~~\ge
\big((b_1\too a_1)*(a_2\too b_2)\big)*
\big((a_1\too b_1)*(b_2\too a_2)\big)\\
&~~~~~~~~~~~~~~~=
\big((b_1\too a_1)*(a_1\too b_1)\big)*
\big((a_2\too b_2)*(b_2\too a_2)\big)\\
&~~~~~~~~~~~~~~~=(a_1\leftrightarrow b_1)*(a_2\leftrightarrow b_2)\\
&~~~~~~~~~~~~~~~=\mathbf{a}\Leftrightarrow\mathbf{b}.
\end{align*}
Therefore, \ref{a12} implies that
$(a_1\too a_2)\leftrightarrow(b_1\too b_2)>r$ which means that is $(b_1\too b_2)\in B_r(a_1\too a_2)\subseteq A$.
Hence $\mathbf{b}\in \too^{-1}(A)$.
\end{proof}
Although the continuous scale $[0,1]_*=([0,1], \wedge, \vee, *, \too, 0, 1)$
endowed to be a BL-algebra and consequently $*$ becomes a continuous t-norm, all the operators of
$[0,1]_*$ are not necessarily continuous with respect to the usual topology on $[0,1]$.
However, Theorem \ref{main1} shows that the introduced topologies $T_*$ and $\mathbf{T}_*$ makes all the operators of any BL-algebra continuous.
\begin{example}\label{luk}
Let $L=[0,1]$. Consider the \g t-norm on $[0,1]$, that is $a * b=\min\{a,b\}$. The residuation relation \ref{res} implies that
$a \rightarrowtail b=\left\{\begin{array}{cc}
1& a\le b\\
b& a>b
\end{array}\right.$ \cite[Theorem 2.1.8]{hajek98}. An easy argument shows that in spite of the continuity of $*$ with respect to the Euclidean topology on $[0,1]$ and $[0,1]^2$, the function $\rightarrowtail$ is not a continuous function. However, Theorem \ref{main1} shows that both $*$ and $\rightarrowtail$ are continuous functions with respect to the topologies $T_*$ and $\mathbf{T}_*$ on $[0,1]$ and $[0,1]^2$.
\end{example}
Now, for any BL-algebra $\mathcal{L}=(L, \wedge, \vee, *, \too , 0, 1)$, we are going to show that the introduced topology $T_*$ makes $\mathcal{L}$ a semitopological BL-algebra.

Recall from \cite{zah-bor-2016} and \cite{bor-rez-kou-2011} that a semitopological
algebra is an algebra $\mathcal{L}=(L, *)$ of type $(2)$ together with a
topology $\tau$ on $L$ such that for all $\delta\in L$ the maps $*^\delta_l: (L, \tau) \to (L, \tau)$ and
$*^\delta_r: (L, \tau)\to (L, \tau)$ defined respectively by
$*^\delta_l(x)= \delta * x$ and $*^\delta_r(x)= x * \delta$ are continuous functions.
\begin{definition}\label{semitop}
Let $\mathcal{L}=(L, \wedge, \vee, *, \too , 0, 1)$ be a BL-algebra. If there exists a
topology $\tau$ on $L$ such that for any $\square\in\{\wedge, \vee, *, \too\}$ and any
$\delta\in L$ the maps $\square^\delta_l: (L, \tau) \to (L, \tau)$ and
$\square^\delta_r: (L, \tau)\to (L, \tau)$ defined respectively by
$\square^\delta_l(x)= \delta \square x$ and $\square^\delta_r(x)= x \square \delta$ are continuous functions, then $(\mathcal{L}, \tau)$ is called a \emph{semitopological BL-algebra}.
\end{definition}
\begin{theorem}\label{semi1}
If $\mathcal{L}=(L, \wedge, \vee, *, \too , 0, 1)$ is a BL-algebra and
$T_*$ is the same as in Theorem \ref{tT}, then $(\mathcal{L},T_*)$ forms a
semitopological BL-algebra.
\end{theorem}
\begin{proof}
Besides \ref{bl4} and \ref{b8} it's enough to clarify that for any $\delta\in L$ the mappings
$*^\delta_l$ ,$*^\delta_r$, $\too^\delta_l$, and $\too^\delta_r$ are
continuous functions. We only do the proof for $*^\delta_l$ and others will be proved in a similar way.

Let $A\in T_*$. We must prove that $(*^\delta_l)^{-1}(A)\in T_*$. So, assume that
$a\in (*^\delta_l)^{-1}(A)$ that is $*^\delta_l(a)\in A$. Hence
$\delta * a\in A$. Since $A$ is an $*$-open set, there exists $r\ll 1$ such that
$B_r(\delta * a)\subseteq A$. We claim that
$B_r(a)\subseteq (*^\delta_l)^{-1}(A)$ which implies that $(*^\delta_l)^{-1}(A)$ is an
$*$-open set and fulfills the proof. To this end, consider an arbitrary element
$b\in B_r(a)$ that is $a\leftrightarrow b>r$. Now, \ref{b15} implies that
\begin{align*}
 (\delta * b)\leftrightarrow(\delta* a) & = (\delta * b)\too (\delta * a)\big)*
\big((\delta * a)\too (\delta * b)\big)\\
 & \ge \big((\delta\too \delta)*(b\too a)\big)
*\big((\delta\too \delta)*(a\too b)\big)\\
 & =  \big(1 *(b\too a)\big)
*\big(1 * (a\too b)\big)\\
 & =  (b\too a)*(a\too b)\\
 & =  a\leftrightarrow b\\
 & > r.
 \end{align*}
So $\delta * b\in B_r(\delta * a)\subseteq A$ means that
$*^\delta_l(b)\in A$. Therefore $b\in (*^\delta_l)^{-1}(A)$ which
completes the proof.
\end{proof}
\section{Some properties of the open ball topology on $[0,1]$}
Example \ref{e2} shows that the $*$-balls are not necessarily $*$-open set.
Furthermore, it shows that $T_*$ does not admit the weakest separation axiom $T_0$. However, when the continuous scale $[0,1]$ endowed to be a BL-algebra, we will show that $T_*$ admits some nice properties.

First of all, the following theorem shows that when $[0,1]$ is endowed to be a BL-algebra, then like  metric spaces, $*$-balls are $*$-open set.
\begin{theorem}
If $[0,1]_*=([0,1],\min,\max, *, \too, 0, 1)$ is a BL-algebra, then, $*$-balls are $*$-open set.
\end{theorem}
\begin{proof}
Note that since $[0,1]$ is linearly ordered, $\ll$ and $<$ have the same meaning. For any $*$-ball $B_r(a)$ and any $b\in B_r(a)$, we must find $\epsilon<1$ such that $B_\epsilon(b)\subseteq B_r(a)$.
To this end, let $\epsilon=(a\leftrightarrow b)\too r$. Since $b\in B_r(a)$, $a\leftrightarrow b>r$ and therefore \ref{b4} implies that
\begin{center}
$\epsilon=(a\leftrightarrow b)\too r<1$.
\end{center}
Now, if $c\in B_\epsilon(b)$ then $b\leftrightarrow c>\epsilon$ that is
%\begin{center}
$b\leftrightarrow c> (a\leftrightarrow b)\too r$.
%\end{center}
Thus, \ref{bl3} implies that $(b\leftrightarrow c)*(a\leftrightarrow b)>r$. Consequently by \ref{l8} and \ref{b1} we get
\begin{center}
$a\leftrightarrow c\ge (a\leftrightarrow b)*(b\leftrightarrow c)
=(b\leftrightarrow c)*(a\leftrightarrow b)>r$
\end{center}
which means that $c\in B_r(a)$. Hence $B_\epsilon(b)\subseteq B_r(a)$ which entails that $B_r(a)$ is an $*$-open set.
\end{proof}
Now, we want to examine the most famous separation axiom for $T_*$ on $[0,1]$ such as a BL-algebra.
\begin{theorem}
If $[0,1]_*=([0,1],\min, \max, *, \too, 0, 1)$ is a BL-algebra, then, $T_*$ would be a Hausdorff topology on $[0,1]$.
\end{theorem}
\begin{proof}
Since $[0,1]_*$ is a BL-algebra, so $*$ is a continuous t-norm on $[0,1]$ (with respect to the usual topology on $[0,1]$). Now, if $a, b$ are two distinct element of $[0,1]$, then there exists $r<1$ that
\begin{equation}\label{ww}
a\leftrightarrow b<r*r.
\end{equation}
Indeed, otherwise $a\leftrightarrow b\ge r*r$ for any $r<1$, which together with the fact that $*$ is a continuous t-norm, implies that $a\leftrightarrow b=1$ and therefore by \ref{l10} $a=b$, a contradiction. To complete the proof, we show that $B_r(a)\cap B_r(b)=\emptyset$. Indeed, if
$c\in B_r(a)\cap B_r(b)$, then \ref{l8} and \ref{ww} leads to the following contradiction.
\begin{center}
$a\leftrightarrow b \ge (a\leftrightarrow c)*(c\leftrightarrow b)\ge r*r>a\leftrightarrow b$.
\end{center}
\end{proof}
%***************************
\section{Describing the open ball topology by means of duality}
In this section, following our conference article \cite{khatami2018}, we show that if one consider a
dual notion of BL-algebras, then the open ball topology could be described by a metric-like topology.

Firstly, consider the following dual notion for BL-algebras.
\begin{definition}\label{dbl-algebra}%\cite{khatami2018}
An \emph{SL-algebra}, is an algebra
$\mathcal{L}=(L, \dota, \dotv, \star, \dotto , 0, 1)$ of type $(2,2,2,2,0,0)$ that satisfies the following conditions.
\begin{enumerate}[label={\normalfont (SL\arabic*)}]
  \item $(L, \dota, \dotv, 0 ,1)$ is a bounded lattice
  with the greatest element $1$ and the smallest element $0$.
  Note that here, $a\le b$ iff $a\dotv b=a$. So, $a\dotv b=\inf\{a,b\}$ and $a\dota b=\sup\{a,b\}$. \label{db1}
  \item $(L, \star, 0)$ is an Abelian monoid,\label{db2}
  \item $\dotto$ is the residua of $\star$, i.e., $a\ge b\dotto c$ iff $a \star b\ge c$ for all $a,b,c\in L$,\label{db3}
  \item $a\dota b=a\star(a\dotto b)$ for all $a,b\in L$,\label{db4}
  \item $(a\dotto b)\dotv(b\dotto a)=0$ for all $a,b\in L$.\label{db5}
\end{enumerate}
\end{definition}
Note that \ref{db3} implies that $b\dotto c=\inf\{a: a\star b\ge c\}$.
\begin{example}
If $S$ is a continuous s-norm, and the residua of $S$ is defined by
$R_S(a,b)=\inf\{c: S(c,a)\ge b\}$,
then
\begin{center}
$[0,1]_S=([0,1], \max, \min, S, R_S, 0, 1)$
\end{center}
forms an SL-algebra.
\end{example}

When $[0,1]_\star=([0,1], \dota, \dotv, \star, \dotto, 0, 1)$ is endowed
to be an SL-algebra, then $\star$ becomes a continuous s-norm, $\dotto$ would be the
residua of $\star$ and therefore $\dota$ and $\dotv$ becomes the maximum and minimum functions, respectively \cite[dual form of Proposition 3]{godo99}.

The following theorem is an obvious consequence of duality between BL-algebras and SL-algebras
which follows from Proposition \ref{blprop}.
\begin{theorem}
Let $\mathcal{L}=(L, \dota, \dotv, \star, \dotto , 0, 1)$ be an SL-algebra.
The following properties hold in $\mathcal{L}$.
%\begin{multicols}{2}
%\begin{enumerate}[label={\thelemma.\arabic*}]
\begin{enumerate}[label={\normalfont (S\arabic*)}]
\item $a\star b=b\star a$ and $(a\star b)\star c=a\star (b\star c)$,\label{d1}
\item $a\star (a\dotto b)\ge b$ and $a\ge b\dotto(a\star b)$,\label{d3}
\item $a\ge b$ iff $a\dotto b=0$,\label{d4}
\item if $a\ge b$ then $a\star c\ge b\star c$, $c\dotto a\ge c\dotto b$, and $a\dotto c\le b\dotto c$,\label{d5}
\item $a\star 1=1$,\label{d2}
\item $(a\dotv b)\star c=(a\star c)\dotv(b\star c)$,\label{d6}
\item $a\star b\ge a$ and $a\ge b\dotto a$,\label{d7}
\item $a\dotv b=\big((a\dotto b)\dotto b\big)\dota\big((b\dotto a)\dotto a\big)$,\label{d8}
\item $(a\dotto b)\ge (b\dotto c)\dotto(a\dotto c)$,\label{d9}
\item $(a\dotto b)\star(b\dotto c)\ge(a\dotto c)$,\label{d10}
\item $a\dotto(b\dotto c)=(a\star b)\dotto c$,\label{d11}
%\item $(a\dotto b)\dotto c\ge\big((b\dotto a)\dotto c\big)\dotto c$,
\item $a\dotto(b\dotto c)=b\dotto(a\dotto c)$,\label{d12}
\item $a\dotto a=0$,\label{d13}
\item $a\dotto b\ge(a\star c)\dotto(b\star c)$,\label{d14}
\item $(a\dotto b)\star (c\dotto d)\ge(a\star c)\dotto(b\star d)$,\label{d15}
\end{enumerate}
%\end{multicols}
\end{theorem}

The key point that we interested in dual of BL-algebras, is that
when $[0,1]_\star=([0,1], \max, \min, \star, \dotto, 0, 1)$ is endowed to be an SL-algebra, then in most cases the dual of the notion $\leftrightarrow$ forms a metric.
\begin{example}\label{luka}
Let $L=[0,1]$, $a\star b=\min\{0,a+b\}$ which is the \lo s-norm, and $\dotto$ be the residua of $\star$. For $a,b\in [0,1]$
\begin{itemize}
  \item if $a\ge b$ then by \ref{d5} and \ref{db2} for any $c\in [0,1]$,
  \begin{center}
  $c\star a\ge c\star b\ge 0\star b=b$
  \end{center}
   and therefore
  \begin{center}
   $a\dotto b=\inf\{c: c\star a\ge b\}=\inf\{c: c\in[0,1]\}=0$,
  \end{center}
  \item if $a< b$ then $c\star a\ge b$ iff $c+a\ge b$ iff $c\ge b-a$ and therefore
  \begin{center}
  $a\dotto b=\inf\{c: c\star a\ge b\}=\inf\{c: c\ge b-a\}=b-a$.
  \end{center}
\end{itemize}
Thus
\begin{center}
$a\dotto b=
\left\{\begin{array}{cc}
0& a\ge b\\
b-a& a<b
\end{array}\right.$.
\end{center}
Now , if $a\ge b$ then $a\dotto b=0$, $b\dotto a=a-b$, and so we have
  \begin{center}
  $(a\dotto b)\star (b\dotto a)=\min\{1,0+a-b\}=a-b$.
  \end{center}
Similarly if $a\le b$ then $(a\dotto b)\star (b\dotto a)=b-a$.
Thus,
\begin{center}
$(a\dotto b)\star (b\dotto a)=|a-b|$
\end{center}
which is the Euclidean metric on $[0,1]$.
\end{example}
\begin{example}\label{godel2}
Let $L=[0,1]$, $a\star b=\max\{a,b\}$ which is the \g s-norm, and $\dotto$ be the residua of $\star$. An argument such as the one in Example \ref{luka} shows that
\begin{center}
$a\dotto b=
\left\{\begin{array}{cc}
0& a\ge b\\
b& a<b
\end{array}\right.$
\end{center}
and
\begin{center}
$(a\dotto b)\star (b\dotto a)=
\left\{\begin{array}{cc}
0& a=b\\
\max\{a,b\}& a\ne b
\end{array}\right.$.
\end{center}
Again, note that $(a\dotto b)\star (b\dotto a)$ is a metric on $[0,1]$.
\end{example}
Certainly we have the following fact.
\begin{theorem}\label{metr}
Assume that $\mathcal{L}=(L, \dota, \dotv, \star, \dotto , 0, 1)$ is an SL-algebra.
Suppose that the mappings $d_\star:L\times L\to L$ is defined by
\begin{center}
$d_\star(a,b)=(a\dotto b)\star(b\dotto a)$.
\end{center}
Then,
\begin{enumerate}
\item $\forall a, b$, $d_\star(a,b)=0$ iff $a=b$,
\item $\forall a, b$, $d_\star(a,b)=d_\star(b,a)$,
\item $\forall a, b, c$, $d_\star(a,b)\le d_\star(a,c)\star d_\star(c,b)$,\label{3}
\item if $L=[0,1]$ and for any $a,b\in[0,1]$, $a\star b\le S_L(a,b)$, then $d_\star$ is a metric on $[0,1]$.
\end{enumerate}
\end{theorem}
\begin{proof}~
\begin{itemize}
  \item [1)] Follows as like as \ref{l10} from \ref{d13}, \ref{d7}, and \ref{d4}.
  \item [2)] Follows as like as \ref{l6} from \ref{d1}.
  \item [3)] Follows as like as \ref{l8} from \ref{d10} and \ref{d5}.
  \item [4)] If $a\star b\le S_L(a,b)$ holds for any $a,b\in[0,1]$, then $(3)$ implies that
%\begin{center}
$d_\star(a,b)\le d_\star(a,c)\star d_\star(c,b)\le d_\star(a,c)+d_\star(c,b)$,
%\end{center}
for any $a,b,c\in[0,1]$, means that $d_\star$ is a metric on $[0,1]$.
\end{itemize}
\end{proof}
A similar argument such as Theorem \ref{metr} holds for the dual notion of $\Leftrightarrow$ which is denoted by $\mathbf{d}_\star$.
\begin{theorem}
Assume that $\mathcal{L}=(L, \dota, \dotv, \star, \dotto , 0, 1)$ is an SL-algebra and define the mappings $\mathbf{d}_\star:L^2\times L^2\to L$ by
\begin{center}
$\mathbf{d}_\star(\mathbf{a},\mathbf{b})=d_\star(a_1,b_1)\star d_\star(a_2,b_2)$.
\end{center}
Then,
\begin{enumerate}
\item $\forall \mathbf{a}, \mathbf{b}$, $\mathbf{d}_\star(\mathbf{a},\mathbf{b})=0$ iff $\mathbf{a}=\mathbf{b}$,
\item $\forall \mathbf{a}, \mathbf{b}$, $\mathbf{d}_\star(\mathbf{a},\mathbf{b})=\mathbf{d}_\star(\mathbf{b},\mathbf{a})$,
\item $\forall \mathbf{a}, \mathbf{b}, \mathbf{c}$,
$\mathbf{d}_\star(\mathbf{a},\mathbf{b})\le \mathbf{d}_\star(\mathbf{a},\mathbf{c})\star
\mathbf{d}_\star(\mathbf{b},\mathbf{c})$,
\item If $L=[0,1]$ and for any $a,b\in[0,1]$, $a\star b\le S_L(a,b)$, then $\mathbf{d}_\star$ is a metric on $[0,1]^2$.
\end{enumerate}
\end{theorem}
\begin{proof}
Similar to the proof of Theorem \ref{metr}.
\end{proof}
Now, for any SL-algebra $\mathcal{L}=(L, \dota, \dotv, \star, \dotto , 0, 1)$, the metric-like topologies on $L$ and $L^2$ could be constructed as the one introduced for BL-algebras in Theorem \ref{tT} which made all the operators of $\mathcal{L}$ continuous.
\begin{theorem}\label{topthe}
Let $\mathcal{L}=(L, \dota, \dotv, \star, \dotto , 0, 1)$ be an SL-algebra. For an
element $a\in L$, write $a\gg 0$ whenever for any $b\in L$, $a\dotv b=0$ implies that $b=0$.
For any $a\in L$ and $\mathbf{a}\in L^2$ and $r\gg 0$, suppose that
%\begin{center}
$N_r(a)=\{b\in L: d_\star(a,b)<r\}$
%\end{center}
and
%\begin{center}
$\mathbf{N}_r(\mathbf{a})=\{\mathbf{b}\in L^2: \mathbf{d}_\star(\mathbf{a},\mathbf{b})<r\}.$
%\end{center}
Then
\begin{center}
$T_\star=\big\{G: G\subseteq L~\text{and}~\forall a\in G,\exists r\gg 0\,\text{such that}\,\big(N_r(a)\subseteq G\big)\big\}$
\end{center}
and
\begin{center}
$\mathbf{T}_\star=\big\{G: G\subseteq L^2~\text{and}~\forall \mathbf{a}\in G,\exists r\gg 0\,\text{such that}\,\big(N_r(\mathbf{a})\subseteq G\big)\big\}$
\end{center}
form topologies on $L$ and $L^2$, respectively. Furthermore,
the mappings $\star:(L^2,\mathbf{T}_\star)\to(L,T_\star)$ and $\dotto:(L^2,\mathbf{T}_\star)\to(L,T_\star)$ are continuous functions.
\end{theorem}
\begin{proof}
Similar to the proof of Theorems \ref{tT} and \ref{main1} with dual notions.
\end{proof}
Now, for any continuous t-norm $*$, if $\star$ is defined by
\begin{center}
$a\star b=1-\big((1-a)*(1-b)\big)$,
\end{center}
then $*$ and $\star$ are called dual and one
can examine the open ball topology on the BL-algebra
$[0,1]_*=([0,1],\min,\max, *, \too, 0, 1)$ by the metric-like topology
on $[0,1]_\star=([0,1], \max, \min, \star, \dotto, 0, 1)$.
\begin{example}
We know that dual of $a*_Lb=\max\{1,a+b-1\}$ is $a\star_L b=\min\{1,a+b\}$. So, the open ball topology on the BL-algebra
\begin{center}
$[0,1]_L^t=([0,1], \min, \max, *_L, \too_L, 0, 1)$
\end{center}
can be examined with the metric topology on the SL-algebra
\begin{center}
$[0,1]_L^s=([0,1], \max, \min, \star_L, \dotto_L, 0, 1)$.
\end{center}
In this special case, the open ball topology on $[0,1]_L^t$ (Example \ref{luk}) and the metric
topology on $[0,1]_L^s$ (Example \ref{luka}) are equivalent. Indeed, replacing any element $b$ with $1-b$ in the $\star$-balls of $([0,1]_L^s,T_\star)$ gives an $*$-ball in $([0,1]_L^t,T_*)$. For example
$N_{0.1}(0.7)=(0.6,0.8)$ corresponds to $B_{0.1}(0.3)=(0.2,0.4)$,
$N_{0.3}(0.1)=[0,0.4)$ corresponds to $B_{0.3}(0.9)=(0.6,1]$, and so forth.
\end{example}
\begin{example}
Let $a*_Gb=\max\{a,b\}$ that is the \g t-norm. We know that
$\too_G=\left\{\begin{array}{ll}
1 & a\le b\\
b & a>b
\end{array}\right.$ \cite[Chapter2]{hajek98}. An argument such as the one in Example \ref{godel2}, shows that
\begin{center}
$(a\too_G b)*_G (b\too_G a)=
\left\{\begin{array}{ll}
1& a=b\\
\min\{a,b\}& a\ne b
\end{array}\right.$.
\end{center}
Therefore, there are two kinds of $*$-balls:
\begin{itemize}
  \item[$r\le a$)]: $B_r(a)=\{b: (a\too_G b)\star_G (b\too_G a)\ge r\}=[r,1]$,
  \item[$r> a$)]: $B_r(a)=\{b: (a\too_G b)\star_G (b\too_G a)\ge r\}=\{a\}$.
\end{itemize}
Note that the only singleton in $[0,1]_G^t$ which is not an $*$-ball is $\{1\}$. So, $T_*$ is a little coarser than the discreet topology on $[0,1]$.

On the other hand, if $a\star_G b=\min\{a,b\}$ which is the \g s-norm, Example \ref{godel2} shows that
\begin{center}
$(a\dotto_G b)\star_G (b\dotto_G a)=
\left\{\begin{array}{cc}
0& a=b\\
\max\{a,b\}& a\ne b
\end{array}\right.$.
\end{center}
So, the $\star$-balls of the $T_\star$ topology on
$[0,1]_G^s$ are as follows:
\begin{itemize}
  \item[$r< a$)]: $N_r(a)=\{b: (a\dotto_G b)\star_G (b\dotto_G a)\le r\}=\{a\}$,
  \item[$r\ge a$)]: $N_r(a)=\{b: (a\dotto_G b)\star_G (b\dotto_G a)\le r\}=[0,r]$.
\end{itemize}
In this case, the only singleton in $[0,1]_G^s$ which is not an $\star$-ball is $\{0\}$.

So, in this case, $T_*$ and $T_\star$ are not equivalent but we could examine
each of them by another. For example,
\begin{itemize}
  \item since for any $a>0$ the singleton $\{a\}$ is open in $T_\star$, by replacing $a$ with $1-a$ we get that for any $a<1$ the singleton $\{a\}$ is open in $T_*$,
  \item since for any $r>0$ the set $[0,r]$ is open in $T_\star$, by replacing any element $b$ with $1-b$ we know that for any $r<1$ the set $[r,1]$ is open in $T_*$,
  \item since $\{\frac{1}{n}\}_{n\in\mathbb{N}}\cup\{0\}$ is a compact subset of $([0,1],T_\star)$, so replacing any element $b$ with $1-b$ leads to the fact that  $\{1-\frac{1}{n}\}_{n\in\mathbb{N}}\cup\{1\}$ is a
  compact subset of $([0,1],T_*)$.
\end{itemize}
\end{example}
\section*{Final remarks}
In this paper we introduced a topology on BL-algebras that makes them semitopological
algebras. One of the advantages of this topology, is the study of
model theoretic properties of Basic logic. In \lo logic the continuity of the interpretation
of logical connectives make it possible to extend some of the results of model theory of classical logic to \lo logic.
However in Basic logic, this study did not developed as like as the \lo logic and the
introduced topology maybe smoothed the future way of this study. Finally, another possible research that maybe
facilitated by the introduced topology, is the study of stone topology for BL-algebras.

%{\bf Acknowledgments:} Author is indebted to Massoud Pourmahdian for his enlightening
%contributions through the preparation of this work.

%% References

\bibliographystyle{alpha}
\bibliography{amin}

\end{document}